\documentclass[11pt]{article}
\usepackage{amssymb}
\newtheorem{precor}{{\bf Corollary}}

\newenvironment{cor}{\begin{precor}{\hspace{-0.5
               em}{\bf.\ }}}{\end{precor}}
\newtheorem{precon}{{\bf Conjecture}}

\newenvironment{con}{\begin{precon}{\hspace{-0.5
               em}{\bf.\ }}}{\end{precon}}
\newtheorem{predefin}{{\bf Definition}}

\newenvironment{defin}[1]{\begin{predefin}{\hspace{-0.5
                   em}{\bf.\ }}{\rm #1}\hfill{$\spadesuit$}}{\end{predefin}}
\newtheorem{preexm}{{\bf Example}}

\newtheorem{preappl}{{\bf Application}}

\newtheorem{prelem}{{\bf Lemma}}

\newenvironment{lem}{\begin{prelem}{\hspace{-0.5
               em}{\bf.\ }}}{\end{prelem}}
\newtheorem{preproof}{{\bf Proof.\ }}

\newenvironment{proof}[1]{\begin{preproof}{\rm
               #1}\hfill{$\blacksquare$}}{\end{preproof}}
\newtheorem{prethm}{{\bf Theorem}}

\newenvironment{thm}{\begin{prethm}{\hspace{-0.5
               em}{\bf.\ }}}{\end{prethm}}
\newtheorem{prealphthm}{{\bf Theorem}}

\newenvironment{alphthm}{\begin{prealphthm}{\hspace{-0.5
               em}{\bf.\ }}}{\end{prealphthm}}
\newtheorem{prealphlem}{{\bf Lemma}}

\newenvironment{alphlem}{\begin{prealphlem}{\hspace{-0.5
               em}{\bf.\ }}}{\end{prealphlem}}
\newtheorem{prepro}{{\bf Proposition}}

\newtheorem{prequ}{{\bf Question}}

\newenvironment{qu}{\begin{prequ}{\hspace{-0.5
               em}{\bf.\ }}}{\end{prequ}}
\newtheorem{preprb}{{\bf Problem}}

\def\conct[#1,#2]{\mbox {${#1} \leftrightarrow {#2}$}}
\def\dconct[#1,#2]{\mbox {${#1} \rightarrow {#2}$}}
\def\deg[#1,#2]{\mbox {$d_{_{#1}}(#2)$}}
\def\mindeg[#1]{\mbox {$\delta_{_{#1}}$}}
\def\maxdeg[#1]{\mbox {$\Delta_{_{#1}}$}}
\def\outdeg[#1,#2]{\mbox {$d_{_{#1}}^{^+}(#2)$}}
\def\minoutdeg[#1]{\mbox {$\delta_{_{#1}}^{^+}$}}
\def\maxoutdeg[#1]{\mbox {$\Delta_{_{#1}}^{^+}$}}
\def\indeg[#1,#2]{\mbox {$d_{_{#1}}^{^-}(#2)$}}
\def\minindeg[#1]{\mbox {$\delta_{_{#1}}^{^-}$}}
\def\maxindeg[#1]{\mbox {$\Delta_{_{#1}}^{^-}$}}
\def\isdef{\mbox {$\ \stackrel{\rm def}{=} \ $}}
\def\dre[#1,#2,#3]{\mbox {${\cal E}_{_{#3}}(#1,#2)$}}
\def\var[#1,#2]{\mbox {${\rm Var}_{_{#1}}(#2)$}}
\def\ls[#1]{\mbox {$\xi^{^{#1}}$}}
\def\hom[#1,#2]{\mbox {${\rm Hom}({#1},{#2})$}}
\def\onvhom[#1,#2]{\mbox {${\rm Hom^{v}}(#1,#2)$}}
\def\onehom[#1,#2]{\mbox {${\rm Hom^{e}}(#1,#2)$}}
\def\core[#1]{\mbox {$#1^{^{\bullet}}$}}
\def\cay[#1,#2]{\mbox {${\rm Cay}({#1},{#2})$}}
\def\cays[#1,#2]{\mbox {${\rm Cay_{s}}({#1},{#2})$}}
\def\dirc[#1]{\mbox {$\stackrel{\rightarrow}{C}_{_{#1}}$}}
\def\cycl[#1]{\mbox {${\bf Z}_{_{#1}}$}}

\date{}
\begin{document}
\footnotetext[1]{This research was partially supported by Shahid
Beheshti University.}
\footnotetext[2]{Correspondence should be addressed to {\tt
hhaji@sbu.ac.ir}.}
\begin{center}
{\Large \bf Graph Powers and Graph Homomorphisms}\\
\vspace*{0.5cm}
{\bf Hossein Hajiabolhassan and Ali Taherkhani}\\
{\it Department of Mathematical Sciences}\\
{\it Shahid Beheshti University,  G.C.,}\\
{\it P.O. Box {\rm 19834}, Tehran, Iran}\\
{hhaji@sbu.ac.ir}\\
{a\_taherkhani@sbu.ac.ir}\\
\end{center}
\begin{abstract}
\noindent In this paper we investigate some basic properties of
fractional powers. In this regard, we show that for any rational
number $1\leq {2r+1\over 2s+1}< og(G)$, $G^{{2r+1\over
2s+1}}\longrightarrow H$ if and only if $G\longrightarrow
H^{-{2s+1\over 2r+1}}.$ Also, for two rational numbers
${2r+1\over 2s+1} < {2p+1\over 2q+1}$ and a non-bipartite graph
$G$, we show that $G^{2r+1\over 2s+1} < G^{2p+1\over 2q+1}$. In
the sequel, we introduce an equivalent definition for circular
chromatic number of graphs in terms of fractional powers. We also
present a sufficient condition for equality of chromatic number and circular chromatic number.\\

\end{abstract}
\section{Introduction}
Throughout this paper we only consider finite graphs. For a graph
$G$, let $V(G)$ and $E(G)$ denote its vertex and edge sets,
respectively. Given two graphs $G$ and $H$, a {\it homomorphism}
from $G$ to $H$ is a map $f: V(G)\longrightarrow V(H)$ such that
adjacent vertices in $G$ are mapped into adjacent vertices in
$H$, i.e., $uv \in E(G)$ implies $f(u)f(v) \in E(H)$. For
simplicity, the existence of a homomorphism is indicated by the
symbol $G \longrightarrow H$. Two graphs $G$ and $H$ are
homomorphically equivalent, denoted by $G \longleftrightarrow
H$,  if $G \longrightarrow H$ and $H\longrightarrow G$. Also, $G
< H$ means that $G \longrightarrow H$ and no homomorphism exists
from $H$ to $G$. In this terminology, we say that $H$ is a bound
for a class ${\cal C}$ of graphs, if $G \longrightarrow H$ for
all $G\in {\cal C}$. The problem of the existence of a bound with
some special properties, for a given class of graphs, has been a
subject of study in graph homomorphism.  A retract of a graph $G$
is a subgraph $H$ of $G$ such that there exists a homomorphism
$r: G \longrightarrow H$, called retraction with $r(u)=u$ for any
vertex $u$ of $H$. A core is a graph which does not retract to a
proper subgraph. Any graph is homomorphically equivalent to a
unique core. Also, the symbol $\hom[G,H]$ is used to denote the
set of all homomorphisms from $G$ to $H$ (for more on graph
homomorphisms
see \cite{DAHA1, DAHA2, HT, HN}).\\

Circular coloring, introduced by Vince \cite{VINCE}, is a model
for coloring the vertices of graphs that provides a more refined
measure of coloring difficulty than the ordinary chromatic
number. If $n$ and $d$ are positive integers with  $n \geq 2d$,
then the {\em circular complete graph} $K_{n\over d}$ is the
graph with vertex set $\{v_{_{0}}, v_{_{1}}, \ldots,
v_{_{n-1}}\}$ in which $v_{_{i}}$ is connected to $v_{_{j}}$ if
and only if $d \leq |i-j| \leq n-d$. A graph $G$ is said to be
$(n, d)$-colorable if $G$ admits a homomorphism to $K_{n\over
d}$. The {\em circular chromatic number} (also known as the {\it
star chromatic number} \cite{VINCE}) $\chi_{_{c}}(G)$ of a graph
$G$ is the minimum of those ratios $\frac{n}{d}$ for which
$gcd(n,d)=1$ and such that $G$ admits a homomorphism to
$K_{n\over d}$. It can be shown that one may only consider
onto-vertex homomorphisms \cite{ZH}. A $(n,d)$-coloring is
''circular'' in the sense that we may view the $n$ colors as
points on a circle, and the requirement for $(n,d)$-coloring is
that the colors on adjacent vertices must be at least $d$
positions apart on the circle. Zhu \cite{ZH} provides a thorough
survey of results on circular chromatic number.

As usual, we denote by $[m]$ the set $\{1, 2, \ldots, m\}$, and
denote by ${[m] \choose n}$ the collection of all $n$-subsets of
$[m]$. The {\em Kneser graph} $KG(m,n)$ is the graph with vertex
set ${[m] \choose n}$, in which $A$ is connected to $B$ if and
only if $A \cap B = \emptyset$. It was conjectured by Kneser
\cite{kne} in 1955, and proved by Lov\'{a}sz \cite{lov} in 1978,
that $\chi(KG(m,n))=m-2n+2$. The {\em Schrijver graph} $SG(m,n)$
is the subgraph of $KG(m,n)$ induced by all $2$-stable $n$-subsets
of $[m]$. It was proved by Schrijver \cite{sch} that
$\chi(SG(m,n))=\chi(KG(m,n))$ and that every proper subgraph of
$SG(m,n)$ has a chromatic number smaller than that of $SG(m,n)$.
Also, for a given graph $G$, the notation $og(G)$ stands for the
odd girth of graph $G$.

For a graph $G$, let $G^{^{k}}$ be the $k$th power of $G$, which
is obtained on the vertex set $V(G)$, by connecting any two
vertices $u$ and $v$ for which there exists a walk of length $k$
between $u$ and $v$ in $G$. Note that the $k$th power of a simple
graph is not necessarily a simple graph itself. For instance, the
$k$th power may have loops on its vertices provided that $k$ is
an even integer. The chromatic number of graph powers has been
studied in the literature (see \cite{BAST, DAHA4, GYJE, HH, SITA,
TA}).

The following simple and useful lemma was proved and used
independently in \cite{DAHA4, NEME1, TA}.

\begin{alphlem}\label{DDD}
\label{M2} Let $G$ and $H$ be two simple graphs such that
$\hom[G,H] \not = \emptyset$. Then, for any positive integer $k$,
$\hom[G^{^{k}},H^{^{k}}] \not = \emptyset$.
\end{alphlem}

Note that Lemma \ref{M2} trivially holds whenever $H^{^{k}}$
contains a loop, e.g., when $k=2$. As immediate consequences of
Lemma~\ref{M2}, we obtain $\chi_c(P)=\chi(P)$ and
$\hom[C,C_{_{7}}]=\emptyset$, where $P$ and $C$ are the Petersen
and the Coxeter graphs, respectively, see  \cite{DAHA4}.

The {\it local chromatic number} of a graph is defined in
\cite{ERFU} as the minimum number of colors that must appear
within distance $1$ of a vertex. For a given graph $G$ with
$odd(G)\geq 7$, the chromatic number of $G^{5}$ provides an upper
bound for local chromatic number of $G$ . In \cite{SITA}, it was
proved if $\chi(G^{5})\leq m$ then $\psi(G) \leq \lfloor {m\over
2}\rfloor +2$.

Now, we recall a definition from \cite{HH}.

\begin{defin}{Let $m, n,$ and $k$ be positive integers with $m\geq 2n$.
Set $H(m,n,k)$  to be the {\it helical graph} whose vertex set
contains all $k$-tuples $(A_1,\ldots ,A_k)$ such that for any
$1\leq r\leq k$, $A_r\subseteq [m], |A_1|=n, |A_r|\geq n$ and for
any $s\leq k-1$ and $t \leq k-2$, $A_s\cap A_{s+1}=\emptyset,
A_t\subseteq A_{t+2}$.
Also, two vertices $(A_1,\ldots ,A_k)$ and $(B_1,\ldots ,B_k)$ of
$H(m,n,k)$ are adjacent if for any $1\leq i, j+1\leq k$, $A_i
\cap B_i=\emptyset, A_j \subseteq B_{j+1}$, and $B_j \subseteq
A_{j+1}$.
}
\end{defin}

 Note that $H(m,1,1)$ is the complete graph $K_m$ and $H(m,n,1)$
is the Kneser graph $KG(m,n)$. It is easy to verify that if $m >
2n$, then the odd girth of $H(m,n,k)$ is greater than or equal to
$2k+1$.

The following theorem shows that the helical graphs are bound of
high odd girth graphs.

\begin{alphthm}{\rm \cite{HH}}\label{HOMB} Let $G$ be a non-empty graph with
odd girth at least $2k+1$. Then, we have
$\hom[G^{(2k-1)},KG(m,n)]\not =\emptyset$ if and only if
$\hom[G,H(m,n,k)] \not = \emptyset.$
\end{alphthm}

Chromatic number of helical graphs has been characterized as
follows.

\begin{alphthm}{\rm \cite{HH}}\label{chrom}
Let $m, n,$ and $k$ be positive integers with $m\geq 2n$. The
chromatic number of the helical graph $H(m,n,k)$ is equal to
$m-2n+2$.
\end{alphthm}

A graph $H$ is said to be a subdivision of a graph $G$ if $H$ is
obtained from  $G$ by subdividing some of the edges. The graph
$S_t(G)$ is said to be the $t$-subdivision of a graph $G$ if
$S_t(G)$ is obtained from $G$ by replacing each edge by a path
with exactly $t-1$ inner vertices. Note that $S_1(G)$ is
isomorphic to $G$.

Hereafter, for a given graph $G$, we will use the following
notation for convenience.
$$G^{{2t+1 \over 2s+1}} \isdef (S_{2s+1}(G))^{2t+1}.$$

For instance, if $n\geq 3$ is a positive integer, then
$K_n^{6r+1\over 2r+1}\simeq K_{rn^2-rn+n}$. It was proved in
\cite{HH}, if $G$ is a graph with odd girth at least $2k+1$, then
a homomorphism from graph $G$ to $(2k+1)$-cycle exists if and
only if the chromatic number of $G^{2k+1\over 3}$ is less than or
equal to 3.

\begin{alphthm}\label{ODDC}{\rm \cite{HH}}
Let $G$ be a graph with odd girth at least $2k+1$. Then,
$\chi(G^{2k+1\over 3}) \leq 3$ if and only if
$\hom[G,C_{2k+1}]\not =\emptyset$.
\end{alphthm}

In what follows we are concerned with fractional powers. The paper
is organized as follows. In second section, we study some basic
properties of fractional power. In this regard, we show that for
any rational number $1\leq {2r+1\over 2s+1}< og(G)$,
$G^{{2r+1\over 2s+1}}\longrightarrow H$ if and only if
$G\longrightarrow H^{-{2s+1\over 2r+1}}.$  Also, for two rational
numbers ${2r+1\over 2s+1} < {2p+1\over 2q+1}$ and a non-bipartite
graph $G$, we show that $G^{2r+1\over 2s+1} < G^{2p+1\over
2q+1}$. In third section, we investigate some basic properties of
power thickness. In fourth section, we introduce an equivalent
definition for circular chromatic number of graphs in terms of
fractional powers. We also present a sufficient condition for
equality of chromatic number and circular chromatic number in
terms of power thickness. Finally, in Section five, we make some
concluding remarks about open problems and natural directions of
generalization.
\section{Fractional Powers}
In this section we investigate the basic properties of graph
powers. The following simple lemma can easily be proved by
constructing graph homomorphisms and its proofs is omitted for
the sake of brevity.

\begin{lem}\label{simple}
Let $G$ be a graph.
\begin{enumerate}
\item[a{\rm )}] If $s$ is a non-negative integer, then $G^{2s+1\over 2s+1}
\longleftrightarrow G$.

\item[b{\rm )}] If $s$ is a non-negative integer where $2s+1 < og(G)$, then
$(G^{2s+1})^{1\over 2s+1} \longrightarrow  G$.
\end{enumerate}
\end{lem}

The next lemma will be useful throughout the paper.

\begin{lem}\label{DUAL}
Let $G$ and $H$ be two graphs where $2s+1 < og(H)$. Then,
$G^{{1\over 2s+1}}\longrightarrow H$ if and only if
$G\longrightarrow H^{2s+1}.$
\end{lem}
\begin{proof}{
Let $G^{{1\over 2s+1}}\longrightarrow H$; then $(G^{{1\over
2s+1}})^{2s+1}\longrightarrow H^{2s+1}$. In view of Lemma
\ref{simple}(a), we have $G\longrightarrow (G^{{1\over
2s+1}})^{2s+1}\longrightarrow H^{2s+1}$. Conversely, assume that
$G\longrightarrow H^{2s+1}$. Hence, $G^{1\over
2s+1}\longrightarrow (H^{2s+1})^{1\over 2s+1}$. On the other
hand, Lemma \ref{simple}(b) shows that $(H^{2s+1})^{1\over
2s+1}\longrightarrow H$, as desired.}
\end{proof}
For given graphs $G$ and $H$ with $v\in V(G)$, set
$$N_i(v)\isdef \{u|{\rm there\ is\ a\ walk\ of\ length}\ i\ {\rm joining}\
u\ {\rm and}\ v\}.$$ Also, for a graph homomorphism $f:
G\longrightarrow H$, define
$$f(N_i(v))\isdef {\displaystyle \bigcup_{u \in N_i(v)}f(u)}.$$

Also, for two subsets $A$ and $B$ of the vertex set of a graph
$G$, we write $A \bowtie B$ if every vertex of $A$ is joined to
every vertex of $B$. Also, for any non-negative integer $s$,
define the graph $G^{-{1\over 2s+1}}$ as follows
$$V(G^{-{1\over 2s+1}})\isdef \{(A_1,\ldots,A_{s+1})|\ A_i\subseteq V(G), |A_1|=1,
\emptyset\not =A_i\subseteq N_{i-1}(A_1) \,\ , i\leq s+1\}.$$

Two vertices $(A_1,\ldots,A_{s+1})$ and $(B_1,\ldots,B_{s+1})$ are
adjacent in $G^{-{1\over 2s+1}}$ if for any $1\leq i\leq s$ and
$1\leq j\leq s+1$, $A_i\subseteq B_{i+1}$, $B_i\subseteq A_{i+1}$,
and $A_{j}\bowtie B_{j}$. Also, for any graph $G$ and ${2s+1\over
2r+1}\leq 1$ define the graph $G^{-{2s+1\over 2r+1}}$ as follows
$$G^{-{2s+1\over 2r+1}}\isdef (G^{-{1\over 2r+1}})^{2s+1}.$$

It is easy to verify that if $r$ is a non-negative integer, then
the odd girth of $G^{-{1\over 2r+1}}$ is greater than or equal to
$2r+3$.
The following theorem is a generalization of Theorem \ref{HOMB}
and Lemma 3(ii) of \cite{TA}.

\begin{thm}\label{GENTHM}
Let $G$ and $H$ be two graphs and $1\leq {2r+1\over 2s+1}< og(G)$.
We have $G^{{2r+1\over 2s+1}}\longrightarrow H$ if and only if
$G\longrightarrow H^{-{2s+1\over 2r+1}}.$
\end{thm}
\begin{proof}{
First, we show that
\begin{equation}\label{homhom}
G^{2r+1}\longrightarrow H\ {\rm if\ and\ only\ if\ }
G\longrightarrow H^{-{1\over 2r+1}}.
\end{equation}
Assume that $g \in \hom[G^{2r+1},H]$. Now, we present a graph
homomorphism, say $f$, from $G$ to $H^{-{1\over 2r+1}}$. If $v$
is an isolated vertex of $G$, then consider an arbitrary vertex,
say $f(v)$, of $H$ as image of $f$. For any non-isolated vertex
$v\in V(G)$, define
$$f(v)\isdef (g(v), g(N_1(v)), g(N_2(v)), \ldots ,g(N_{r}(v))).$$

Since $g$ is a graph homomorphism from $G^{2r+1}$ to $H$, one can
verify that for any vertex $v\in V(G)$, $f(v) \in H^{-{1\over
2r+1}}$. Also, for any $0\leq i, j+1 \leq r$, we have $g(N_i(v))
\bowtie g(N_i(u)),\ g(N_j(v)) \subseteq g(N_{j+1}(u))$, and
$g(N_j(u)) \subseteq g(N_{j+1}(v))$ provided that  $u$ is adjacent
to $v$. Hence, $f$ is a graph homomorphism from $G$ to
$H^{-{1\over 2r+1}}$.

Next, let $\hom[G,H^{-{1\over 2r+1}}] \not = \emptyset$ and $f: G
\longrightarrow H^{-{1\over 2r+1}}$. Assume $v\in V(G)$ and
$f(v)=(A_1,A_2,\ldots ,A_{r+1})$. Define, $g(v)\isdef A_1$. We
show that $g\in \hom[G^{2r+1},H]$. Assume further that $u,v \in
V(G)$ such that there is a walk of length $2t+1$ ($t\leq r$)
between $u$ and $v$ in $G$, i.e., $uv\in E(G^{2r+1})$. Consider
adjacent vertices $u'$ and $v'$ such that $u'\in N_t(u)$ and
$v'\in N_t(v)$. Also, let $f(v)=(A_1,A_2,\ldots ,A_{r+1})$,
$f(u)=(B_1,B_2,\ldots ,B_{r+1})$, $f(v')=(A'_1,A'_2,\ldots
,A'_{r+1})$, and $f(u')=(B'_1,B'_2,\ldots ,B'_{r+1})$. In view of
the definition of $H^{-{1\over 2r+1}}$, we obtain $A_1 \subseteq
A'_{t+1}$ and $B_1 \subseteq B'_{t+1}$. On the other hand,
$A'_{t+1}\bowtie B'_{t+1}$, which yields $g(v)$ is adjacent to
$g(u)$. Thus, $\hom[G^{2r+1},H]\not =\emptyset$.

Now, assume that $G^{{2r+1\over 2s+1}}\longrightarrow H$. In view
of (\ref{homhom}), one has $G^{{1\over 2s+1}}\longrightarrow
H^{-{1\over 2r+1}}$; consequently, $G\longrightarrow (G^{{1\over
2s+1}})^{2s+1}\longrightarrow (H^{-{1\over 2r+1}})^{2s+1}$.
Conversely, suppose $G\longrightarrow H^{-{2s+1\over 2r+1}}$.
Considering Lemma \ref{DUAL}, we have $G^{1\over 2s+1}
\longrightarrow H^{-{1\over 2r+1}}$. Now, in view of
(\ref{homhom}), one can conclude that $G^{2r+1\over 2s+1}
\longrightarrow H$. }
\end{proof}

Note that in Theorem \ref{GENTHM} we assume that $1\leq {2r+1\over
2s+1}$, since $2s+1$ should be less than the odd girth of
$H^{-{1\over 2r+1}}$. In fact, we don't know the exact value of
$og(H^{-{1\over 2r+1}})$. Though, we specify the odd girth of
$H(m,1,k)$ in Lemma \ref{Oddg}. This can be used to generalize
Theorem \ref{GENTHM}. Also, it should be noted that the above
theorem, for the case $s=0$, was obtained by C. Tardif (personal
communication).

\begin{cor}\label{GENHEL}
Let $G$  be a non-bipartite graph. If $1 \leq {2r+1 \over 2s+1}
 < og(G)$, then $G^{{2r+1 \over 2s+1}} \longrightarrow
K_m$ if and only if $G \longrightarrow H(m,1,r+1)^{2s+1}$.
\end{cor}

\begin{lem}\label{INV}
Let $G$ be a non-bipartite graph.  For any non-negative integer
$r$ we have
$$(G^{-{1\over 2r+1}})^{2r+1} \longleftrightarrow G.$$
\end{lem}
\begin{proof}{
First, note that $G^{-{1\over 2r+1}} \longrightarrow G^{-{1\over
2r+1}}$. Hence, in view of Theorem \ref{GENTHM}, we have
$(G^{-{1\over 2r+1}})^{2r+1} \longrightarrow G$. Next,
$G^{2r+1\over 2r+1}\longrightarrow G$. Considering Theorem
\ref{GENTHM}, we have $G^{1\over 2r+1}\longrightarrow G^{-1\over
2r+1}$. Thus, $G\longrightarrow (G^{1\over
2r+1})^{2r+1}\longrightarrow (G^{-1\over 2r+1})^{2r+1}$, as
required.}
\end{proof}

However, in general, $(G^{2r+1})^{-{1\over 2r+1}}$ and $G$ are
quite different. For example, $(C_5^{3})^{-{1\over
3}}=K_5^{-{1\over 3}}$ is not homomorphically equivalent to
$C_5$. In fact, $\chi(K_5^{-{1\over 3}})=\chi(H(5,1,2))=5$, while
$\chi(C_5)=3$. Also, it should be noted that for given positive
integers $k$, $m$, and $n$ where $m >2n$, the helical graph
$H(m,n,k)$ and the graph $KG(m,n)^{-{1\over 2k-1}}$ are
homomorphically equivalent. Although, if $k\geq 2$ and $n\geq 2$,
then the number of vertices of $H(m,n,k)$ is less than that of
$KG(m,n)^{-{1\over 2k-1}}$.

We introduce some notation used for the remainder of the paper.
Let $G$ be a graph which does not contain isolated vertices. Set
the vertex set of $G^{1 \over 2s+1}$ as follows. By abuse of
notation, for any edge $uv \in E(G)$, define $(uv)_0\isdef u$ and
$(vu)_0\isdef v$. Note that a vertex may have several
representations. Moreover, $(2s+1)$st subdivision of the edge $uv$
is a path of length $2s+1$, say $P_{uv}$,  set the vertices and
the edges of this path, respectively, as follows

$$V(P_{uv})\isdef \{(uv)_0,(uv)_1,\ldots ,(uv)_s,(vu)_0, (vu)_1,\ldots ,(vu)_s \}$$
and
$$E(P_{uv})\isdef \{(uv)_i(vu)_{s-i},\ (vu)_{s-j+1}(uv)_{j}|\ 0\leq i \leq s,\ 1\leq j \leq s\}.$$

Also, note that the graph $G^{2r+1 \over 2s+1}$ is $(2r+1)$st
power of $G^{1 \over 2s+1}$. Hence, we follow the aforementioned
notation for the vertex set of $G^{2r+1 \over 2s+1}$.

For a given non-bipartite graph $G$, we are going to prove a
density theorem for family of its fractional powers. First, we
prove an auxiliary lemma.

\begin{lem}\label{KLEM}
Let $G$  be a non-bipartite graph.

\begin{enumerate}
\item [a{\rm )}] If ${(2r+1)(2p+1) \over
(2s+1)}< og(G)$, then $G^{{(2r+1)(2p+1) \over (2s+1)}}
\longleftrightarrow (G^{{2r+1 \over 2s+1}})^{{2p+1}}.$

\item [b{\rm )}] If ${2r+1\over 2s+1}< og(G)$ and ${(2r+1)(2p+1)\over (2s+1)(2q+1)} < og(G)$, then $(G^{{2r+1\over
2s+1}})^{2p+1\over 2q+1}\longrightarrow G^{{(2r+1)(2p+1)\over
(2s+1)(2q+1)}}.$
\end{enumerate}
\end{lem}
\begin{proof}{
Part (a) follows by a simple argument. To prove part (b), note
that
$$
\begin{array}{lll}
 & G^{2r+1\over 2s+1} \longrightarrow G^{(2r+1)(2q+1)\over
(2s+1)(2q+1)}&  \\
\Rightarrow  & G^{2r+1\over 2s+1}   \longrightarrow
(G^{(2r+1)\over(2s+1)(2q+1)})^{2q+1}& (\rm{by\ Lemma\ \ref{KLEM}(a) )}\\

\Rightarrow & (G^{2r+1\over 2s+1})^{1\over 2q+1} \longrightarrow
G^{{(2r+1)\over (2s+1)(2q+1)}} &  (\rm{by\ Lemma\ \ref{DUAL})}\\
\Rightarrow & ((G^{2r+1\over 2s+1})^{1\over 2q+1})^{2p+1}
\longrightarrow (G^{{(2r+1)\over (2s+1)(2q+1)}})^{2p+1} &
(\rm{by\ Lemma\ \ref{DDD})}\\
\Rightarrow & (G^{2r+1\over 2s+1})^{2p+1\over
2q+1}\longrightarrow G^{{(2r+1)(2p+1)\over (2s+1)(2q+1)}} & (\rm{
by\ Lemma\ \ref{KLEM}(a))}\\
\end{array}
$$
}
\end{proof}

One important property of the family of circular complete graph is
that $K_{r\over s} < K_{p\over q}$ if and only if ${r\over s} <
{p\over q}$. Fortunately, for a given non-bipartite graph $G$, we
have a similar property for the family of fractional powers of
$G$.

\begin{thm}\label{TICK}
Let $G$  be a non-bipartite graph. If $ {2r+1 \over 2s+1} < {2p+1
\over 2q+1} < og(G)$, then $$G^{{2r+1 \over 2s+1}} < G^{{2p+1
\over 2q+1}}.$$
\end{thm}
\begin{proof}{
First, we show that if $1 < {2r+1 \over 2s+1}
 < og(G)$, then $G < G^{{2r+1 \over 2s+1}}.$
 We know that $G \longrightarrow G^{{2r+1 \over 2s+1}}.$ Hence, it is
sufficient to show that there is no homomorphism from $G^{{2r+1
\over 2s+1}}$ to $G$. First, we prove that if $G$ is a core, then
the statement is true. On the contrary, suppose that there is a
homomorphism from $G^{{2r+1 \over 2s+1}}$ to $G$. Since, $G$ is
core and induced subgraph of $G^{{2r+1 \over 2s+1}}$, this
homomorphism provides an isomorphism between two copies of $G$.
For any edge of $G$, say $e=uv$, The vertex $(uv)_1$ (resp.
$(vu)_1$) of $G^{{2r+1 \over 2s+1}}$  is adjacent to all
neighborhood of the vertex $u$ (resp. $v$). $G$ is a core;
therefore, the image of $(uv)_1$ (resp. $(vu)_1$) should be the
same as $u$ (resp. $v$). By induction, one can show that image of
$(uv)_{k}$ (resp. $(vu)_{k}$) should be the same as $u$ (resp.
$v$) whenever $1 \leq k \leq s$. Now, note that since $G$ is a
non-bipartite graph; hence, it contains a triangle or an induced
path of length three. Assume that $G$ contains a triangle with
vertex set $\{u, v, w\}$. Consider two vertices $(uv)_{s}$ and
$(uw)_{s}$. It was shown that images of $(uv)_{s}$ and $(uw)_{s}$
should be $u$. Also, $1 < {2r+1 \over 2s+1}$; consequently,
$(uv)_{s}$ and $(uw)_{s}$ are adjacent which is a contradiction.
Similarly, if $G$ contains an induced path of length three, we
get a contradiction.

Now, suppose that $G$ is an arbitrary non-bipartite graph. It is
well-know that $G$ contains a core, say $H$, as induced subgraph.
On the contrary, suppose that $G^{{2r+1 \over 2s+1}}
\longrightarrow G.$ Then, we have $H^{{2r+1 \over 2s+1}}
\longrightarrow G^{{2r+1 \over 2s+1}} \longrightarrow G
\longrightarrow H,$ which is a contradiction. Consequently, if $1
< {2r+1 \over 2s+1} < og(G)$, then $G < G^{{2r+1 \over 2s+1}}.$

It is readily to check that
$$
G^{{2r+1 \over 2s+1}}\longleftrightarrow G^{{(2r+1)(2q+1) \over
(2s+1)(2q+1)}}\quad {\rm and} \quad G^{{2p+1 \over 2q+1}}
\longleftrightarrow G^{{(2p+1)(2s+1) \over (2q+1)(2s+1)}}.
$$
On the other hand, we have ${2r+1 \over 2s+1} < {2p+1 \over
2q+1}$; hence, $G^{{(2r+1)(2q+1) \over (2s+1)(2q+1)}}
\longrightarrow G^{{(2p+1)(2s+1) \over (2q+1)(2s+1)}}$. It
remains to show that the inequality is strict. On the contrary,
assume that $G^{{2p+1 \over 2q+1}} \longrightarrow G^{{2r+1 \over
2s+1}}.$ Then, in view of Lemma \ref{KLEM}(b) we have
$$(G^{{2p+1 \over 2q+1}})^{{(2s+1)(2p+1) \over (2r+1)(2q+1)}}
\longrightarrow (G^{{2r+1 \over 2s+1}})^{{(2s+1)(2p+1) \over
(2r+1)(2q+1)}}\longrightarrow G^{{2p+1 \over 2q+1}}.$$
Note that ${{(2s+1)(2p+1) \over (2r+1)(2q+1)}} >1$ which is a
contradiction, as desired.}
\end{proof}
\section{Power Thickness}
Considering Theorem \ref{ODDC}, it is worth studying the following
definition which has been introduced in \cite{HH}.

\begin{defin}{ Assume that $G$ is a non-bipartite
graph. Also, let $i \geq -\chi(G)+3$ be an integer. {\it $i$th
power thickness} of $G$ is defined as follows.
$$\theta_i(G) \isdef \sup\{{2r+1 \over 2s+1}| \chi(G^{{2r+1 \over 2s+1}})\leq \chi(G)+i,
{2r+1 \over 2s+1}< og(G) \}.$$ For simplicity, when $i=0$, the
$0$th power thickness of $G$ is called power thickness of $G$ and
it is denoted by $\theta(G)$.}
\end{defin}

The importance of power thickness is that it allows us to obtain
necessary condition for the existence of graph homomorphisms.

\begin{lem}\label{NOHOM}
Let $G$ and $H$  be two non-bipartite graphs with
$\chi(G)=\chi(H)-j,\ j\geq 0$. If $G \longrightarrow H$ and $i+j
\geq -\chi(G)+3$, then
$$\theta_{i+j}(G) \geq \theta_i(H).$$
\end{lem}
\begin{proof}{
Consider a rational number ${2r+1 \over 2s+1}< og(H)$ for which
$\chi(H^{{2r+1 \over 2s+1}}) \leq \chi(H)+i$. We know that $og(G)
\geq og(H)$ since $G\longrightarrow H$. Hence, ${2r+1 \over
2s+1}< og(G)$ and $G^{{2r+1 \over 2s+1}} \longrightarrow H^{{2r+1
\over 2s+1}}$ which implies that $\chi(G^{{2r+1 \over 2s+1}}) \leq
\chi(H)+i=\chi(G)+i+j$.}
\end{proof}

In view of Theorem \ref{ODDC}, it is a hard task to compute the
power thickness of arbitrary graphs. Hereafter, we will introduce
some results in this regard. Finding graphs with high power
thickness arises naturally in the mind. In this direction, we
compute the power thickness of some helical graphs.

\begin{thm}\label{heli}
Let $k$, $l$, and $m$ be  positive integers where $m\geq 3$ and
${2l-1 \over 2k-1}\leq 1$. Then,
$$\theta(H(m,1,k)^{2l-1})={2k-1\over 2l-1}$$.
\end{thm}
\begin{proof}{
In view of Lemma \ref{KLEM}(b) and Theorem \ref{HOMB}, we have
$(H(m,1,k)^{2l-1})^{2k-1 \over 2l-1} \longrightarrow
H(m,1,k)^{2k-1} \longrightarrow K_m$; therefore,
$\theta(H(m,1,k)^{2l-1})\geq {2k-1\over 2l-1}$. Suppose, on the
contrary, that $\theta(H(m,1,k)^{2l-1})=t > {2k-1\over 2l-1}$.
Choose a rational number $1 < {2r+1\over 2s+1}$ such that $1 <
{(2r+1)(2k-1) \over (2s+1)(2l-1)} < t$. Set $G\isdef
(H(m,1,k)^{2l-1})^{2r+1\over 2s+1}$. In view of Lemma
\ref{KLEM}(b) and definition of power thickness, one has
$\chi(G^{2k-1\over 2l-1}) \leq m$. By Theorem \ref{GENTHM}, one
has $G \longrightarrow H(m,1,k)^{2l-1}$. Thus,
$(H(m,1,k)^{2l-1})^{2r+1 \over 2s+1} \longrightarrow
H(m,1,k)^{2l-1}$ which contradicts Theorem \ref{TICK}, as
claimed.}
\end{proof}

The next definition provides a sufficient condition for the graphs
with $\theta(G)=1$.

\begin{defin}{Let $G$ be a graph with chromatic number $k$. $G$ is
called a colorful graph if for any $k$-coloring $c$ of $G$, there
exists an induced subgraph $H$ of $G$ such that for any vertex
$v$ of $H$, all colors appear in closed neighborhood of $v$,
i.e., $c(N[v]) = \{1,2,\ldots ,k\}$.}
\end{defin}

\begin{thm}\label{colorful}
For any non-bipartite colorful graph $G$, we have $\theta(G)=1$.
\end{thm}
\begin{proof}{
On the contrary, suppose that $\theta(G)>1.$ Choose a rational
number $1 < {2r+1 \over 2s+1} < \theta(G)$. By definition,
$\chi(G^{{2r+1 \over 2s+1}})=\chi(G)=k$. Consider a $k$-coloring
of the graph $G^{{2r+1 \over 2s+1}}$. Since, $G$ is a colorful
graph and an induced subgraph of $G^{{2r+1 \over 2s+1}}$, there
exists an induce subgraph of $G^{{2r+1 \over 2s+1}}$, denoted by
$H$, such that for any vertex $v$ of $H$, all colors appear in
closed neighborhood of $v$. For any edge of $H$, say $e=uv$, the
vertex $(uv)_1$ (resp. $(vu)_1$) of $G^{{2r+1 \over 2s+1}}$  is
adjacent to all neighborhood of the vertex $u$ (resp. $v$).
Therefore, the color of $(uv)_1$ (resp. $(vu)_1$) should be the
same as $u$ (resp. $v$). By induction, one can show that the color
of $(uv)_{k}$ (resp. $(vu)_{k}$) should be the same as $u$ (resp.
$v$) provided that $uv\in E(H)$. In view of coloring property of
$H$, it should contain a triangle or an induced path of length
three whose end vertices have the same color. Assume that $H$
contains an induced path with vertex set $\{u,v,w,x\}$ and edge
set $\{uv, vw, wx\}$ such that $u$ and $x$ have the same color.
Consider two vertices $(uv)_{s}$ and $(xw)_{s}$. It was shown
that colors of $(uv)_{s}$ and $(xw)_{s}$ should be the same as
$u$ and $x$, i.e, they have the same color. On the other hand, $1
< {2r+1 \over 2s+1}$; consequently, $(uv)_{s}$ and $(xw)_{s}$ are
adjacent which is a contradiction. Similarly, if $H$ contains a
triangle, we get a contradiction.}
\end{proof}

We know that any uniquely colorable graph is a colorful graph.
Hence, the power thickness of non-bipartite uniquely colorable
graphs is one.

\begin{cor}
Let $K_n$ be complete graph with $n\geq 3$ vertices. Then,
$\theta(K_n)=1$
\end{cor}

A less ambitious objective is to find all graphs with power
thickness one. Also, we don't know whether any graph with power
thickness one is colorful.

Of particular interest is the conclusion that circular complete
graph  $K_{2n+1 \over n-t}$ is isomorphic to $C_{2n+1}^{2t+1}$.
This allows us to investigate some coloring properties of
circular complete graph powers.

\begin{lem}\label{chromc}
Given non-negative integers $n$ and $t$ where $n> t$. We have
\begin{enumerate}
\item[{\rm a)}] $C_{2n+1}^{2t+1}\simeq K_{2n+1 \over n-t}$

\item[{\rm b)}] $\theta(C_{2n+1})={2n+1\over 3}.$
\end{enumerate}
\end{lem}
\begin{proof}{Part (a) follows by a simple discussion.
Note that $C_{2n+1}^{2r+1\over 2s+1}$ and$C_{(2n+1)(2s+1)}^{2r+1}$
are isomorphic. Also, in view of part (a),
$\chi(C_{2n+1}^{2r+1})=\lceil {2n+1\over n-r} \rceil$. Now, part
(b) follows by part (a).}
\end{proof}

Now, we are ready to specify the odd girth of $H(m,1,k)$.

\begin{lem}\label{Oddg}
Let $m \geq 3$ and $k$ be positive integers. The odd girth of  the
helical graph $H(m,1,k)$ is equal to $2k+2\lceil{2k-1\over m-2}
\rceil-1$.
\end{lem}
\begin{proof}{
Let $C_{2n+1} \longrightarrow H(m,1,k)$; then $C_{2n+1}^{2k-1}
\longrightarrow H(m,1,k)^{2k-1}$. On the other hand, we know that
$\chi(H(m,1,k)^{2k-1})=m$; consequently, in view of Lemma
\ref{chromc}(a) we have $\lceil{2n+1 \over n-k+1} \rceil \leq m$.
Also, if $\lceil{2n+1 \over n-k+1} \rceil \leq m$, then by using
Lemma \ref{chromc}(a) we have $C_{2n+1}^{2k-1} \longrightarrow
K_m$ which this implies that $C_{2n+1} \longrightarrow H(m,1,k)$.
Therefore, the odd girth of the helical graph $H(m,1,k)$ is the
smallest value of $2n+1$ for which $\lceil{2n+1 \over n-k+1}
\rceil \leq m$. It is easy to check that the odd girth of
$H(m,1,k)$ should be $2k+2\lceil{2k-1\over m-2} \rceil-1$.}
\end{proof}
\section{Circular Coloring}
The remainder of this paper is devoted to connection between
chromatic number of graph powers and circular coloring. In the
next theorem we introduce an equivalent definition for circular
chromatic number of graphs.

\begin{thm}\label{defchic}
Let G be a non-bipartite graph with chromatic number $\chi(G)$.
Then, $\chi(G)\neq\chi_c(G)$ if and only if there exists a
rational number ${{2r+1}\over
{2s+1}}>{\frac{\chi(G)}{3(\chi(G)-2)}}$ for which
$\chi(G^{{2r+1}\over {2s+1}})= 3$. Moreover, $\chi_c(G)=\inf
\{{2n+1\over n-t}| \chi(G^{2n+1\over 3(2t+1)})=3, n> t >0\}.$
Also, if ${{2r+1}\over {2s+1}} \leq
{\frac{\chi(G)}{3(\chi(G)-2)}}$, then $\chi(G^{{2r+1}\over
{2s+1}})=3$.
\end{thm}
\begin{proof}{First, assume that $\chi(G)\neq\chi_c(G)$ and $\chi_c(G) <
{2d\chi(G)-1 \over 2d} < \chi(G)$ where $d$ is sufficiently
large. By Lemma \ref{chromc}(a), $C_{2n+1}^{2t+1} \simeq K_{2dx-1
\over 2d}$ whenever $n=d\chi-1$ and $t=d(\chi-2)-1$. Therefore,
$G\longrightarrow K_{2n+1\over{n-t}}.$ On the other hand,
\begin{equation}\label{chichic}
\begin{array}{llll}
 G\longrightarrow K_{2n+1\over n-t} & \Longleftrightarrow &
G\longrightarrow
 K_{2n+1\over n}^{2t+1} & ({\rm by\ Lemma\
 \ref{chromc}}(a)) \\
 & \Longleftrightarrow & G^{1\over{2t+1}}\longrightarrow
K_{2n+1\over n} & ({\rm{by\ Lemma}\ \ref{DUAL})}\\
 & \Longleftrightarrow & G^{1\over{2t+1}}\longrightarrow
  C_{2n+1} & \\
 &\Longleftrightarrow & \chi(G^{2n+1\over{3(2t+1)}})= 3 & ({\rm{by\ Theorem\  \ref{ODDC}}})
\end{array}
\end{equation}
It is readily seen that
${2n+1\over{3(2t+1)}}>{\frac{\chi(G)}{3(\chi(G)-2)}}$ whenever
$n=d\chi-1$ and $t=d(\chi-2)-1$. Hence, it suffices to set
${{2r+1}\over {2s+1}}={2n+1\over{3(2t+1)}}$; consequently,
$\chi(G^{{2r+1}\over {2s+1}})= 3$. Conversely, let
$\chi(G^{{2r+1}\over {2s+1}})= 3$ for a rational number
${{2r+1}\over {2s+1}}>{\frac{\chi(G)}{3(\chi(G)-2)}}.$ Choose
positive integers $n$ and $t$ which satisfy
${\frac{\chi(G)}{3(\chi(G)-2)}} < {2n+1\over{3(2t+1)}} \leq
{{2r+1}\over {2s+1}}$. By Theorem \ref{TICK},
$G^{2n+1\over{3(2t+1)}}\longrightarrow G^{{2r+1}\over {2s+1}}$
and so $3\leq \chi(G^{2n+1\over{3(2t+1)}})\leq \chi(
G^{{2r+1}\over {2s+1}})$. In view of (\ref{chichic}) we have
$G\longrightarrow K_{2n+1\over {n-t}}$ provided that
$\chi(G^{2n+1\over{3(2t+1)}})=3$. Also, it is straightforward to
verify ${2n+1\over{n-t}}<\chi(G)$ whenever
${\frac{\chi(G)}{3(\chi(G)-2)}}<{2n+1\over{3(2t+1)}}$. Thus,
$\chi_c(G)<\chi(G).$

Also, the (\ref{chichic}) shows that $\chi_c(G)=\inf \{{2n+1\over
n-t}| \chi(G^{2n+1\over 3(2t+1)})=3\}.$ Finally, suppose that
${{2r+1}\over {2s+1}} \leq {\frac{\chi(G)}{3(\chi(G)-2)}}$. As
before, choose positive integers $n$ and $t$ which satisfy
${{2r+1}\over {2s+1}}\leq {2n+1\over{3(2t+1)}} \leq
{\frac{\chi(G)}{3(\chi(G)-2)}}$. One can see that
${2n+1\over{n-t}}\geq \chi(G)$, and therefore $G\longrightarrow
K_{2n+1\over n-t}$. Now, (\ref{chichic}) implies that
$\chi(G^{2n+1\over 3(2t+1)})=3$. By Theorem \ref{TICK}, we have
$3 \leq \chi(G^{2r+1\over 2s+1}) \leq \chi(G^{2n+1\over
3(2t+1)})=3$, as claimed.
}
\end{proof}

We show that the power thickness of circular complete graphs
$K_{p\over q}$ is greater than one provided that $q\nmid p$.
\begin{thm}\label{CIRCOM} For any rational number ${p\over
q}>2$ where $q\nmid p$ we have
$$\theta(K_{p\over q})> 1.$$
\end{thm}
\begin{proof}{
Set $m\isdef {\lceil\frac{p}{q}\rceil}$. Choose a positive integer
$d$ such that ${p\over q} < {\frac{2d{m}-1} {2d}}<m $. We know
that $K_{p \over q} \longrightarrow K_{\frac{2d{m}-1} {2d}}$;
hence, it is sufficient to show that there exists  a positive
integer $s$ such that  $(K_{\frac{2d{m}-1} {2d}})^{{2s+1}\over
{2s-1}}\longrightarrow K_m$.

Set $n \isdef {d m-1}$, $t\isdef d(m-2)-1$. In view of Lemma
\ref{chromc}(a) and Lemma \ref{KLEM}(b), we have
$$(K_{\frac{2d{m}-1} {2d}})^{{2s+1}\over
{2s-1}}\simeq (C_{2n+1}^{2t+1})^{\frac{2s+1}{2s-1}}\longrightarrow
(C_{2n+1})^{\frac{(2t+1)(2s+1)}{2s-1}}\simeq
(C_{(2n+1)(2s-1)})^{(2t+1)(2s+1)}.$$
On the other hand, Lemma \ref{chromc}(b) confirms that $$\chi(
(C_{(2n+1)(2s-1)})^{(2t+1)(2s+1)})=\lceil
\frac{(2n+1)(2s-1)}{(n-t)(2s+1)-2n-1}\rceil.$$ Therefore,
$$\chi((K_{\frac{2d{m}-1} {2d}})^{{2s+1}\over {2s-1}})\leq \lceil
\frac{(2dm-1)(2s-1)}{2d(2s+1)-2dm+1}\rceil.$$
It is easily to see that if $s$ is sufficiently large, then
$\chi((K_{\frac{2d{m}-1} {2d}})^{{2s+1}\over {2s-1}})= m.$ In
other words, $\theta(K_{p\over q})\geq {2s+1\over{2s-1}}>1.$}
\end{proof}

The aforementioned theorem provides a sufficient condition for
equality of chromatic number and circular chromatic number of
graphs. In fact, if we show that power thickness of a graph $G$ is
equal to one, then $\chi(G)=\chi_c(G)$.

In case $\chi(G)=3$, it is well-known that $\chi_c(G)=3$ if and
only if $G$ is a colorful graph.

\begin{thm}\label{circular}
Let $G$ be a graph with chromatic number $3$. Then, $\theta(G)=1$
if and only if $\chi_c(G)=3$.
\end{thm}
The problem whether the circular chromatic number and the
chromatic number of the Kneser graphs and the Schrijver graphs are
equal has received attention and has been studied in several
papers \cite{DAHA3, HAZH, jhs, LILI, ME, SITA}. Johnson, Holroyd,
and Stahl \cite{jhs} proved that $\chi_c({\rm KG}(m,n))=\chi({\rm
KG}(m,n))$ if $m\leq 2n+2$ or $n=2$. This shows ${\rm
KG}(2n+1,n)$ is a colorful graph.

\begin{cor}
Let $n$ be a positive integer. Then, $\theta({\rm KG}(2n+1,n))=1$
\end{cor}

They also conjectured that the equality holds for all Kneser
graphs.

\begin{con}
\label{jhsconj} {\rm \cite{jhs}} For all $m \geq 2n+1$,
$\chi_c({\rm KG}(m,n))=\chi({\rm KG}(m,n))$.
\end{con}

\begin{qu}\label{CF}
Given positive integers $m$ and $n$ where $m\geq 2n$, is the
Kneser graph ${\rm KG}(m,n)$ a colorful graph? Is it true that
$\theta({\rm KG}(m,n))=1$?
\end{qu}

Theorem \ref{HOMB} shows that $\theta(H(m,n,k))\geq 2k-1$
whenever $m\geq 2n+1$. Another problem which may be of interest
is the following.

\begin{qu}\label{TK}
Given positive integers $m$ and $n$ where $m\geq 2n+1$, is it true
that $\theta(H(m,n,k))=2k-1$?
\end{qu}

Odd cycles are symmetric and they have sparse structure. Hence,
it can be useful if circular chromatic number can be expressed as
homomorphism to odd cycles. Now, let $G$ be a non-bipartite graph
and $t$ be a positive integer. Define,
$$f(G,2t+1)\isdef \max \{2n+1 | G^{1\over 2t+1}\longrightarrow
C_{2n+1}\}.$$

One can see that $3\leq f(G,2t+1) \leq (2t+1)\times og(G)$. In
view of proof of Theorem \ref{defchic}, one can compute
$f(G,2t+1)$ in terms of circular chromatic number of graph $G$
and vice versa. In fact, we have
$$\chi_c(G)=\inf\{{2n+1\over n-t}|
G^{1\over{2t+1}}\longrightarrow C_{2n+1}, n> t >0\}.$$ Moreover,
$$f(G,2t+1)=2\lfloor {1+ t \chi_c(G)\over \chi_c(G)-2}
\rfloor+1.$$

Also, note that there exists an necessary condition for the
existence of homomorphism to symmetric graphs in terms of
eigenvalue of Laplaican matrix. The next theorem can be useful in
studying circular chromatic number of graphs.

\begin{alphthm}{\rm \cite{DAHA1,DAHA2,DAHA3}}\protect\label{ALGH}
Let $G$ be a graph with $|V(G)|=m$. If $\sigma  \in
\hom[G,C_{2n+1}]$, then,
$$\lambda^{^{G}}_{_{m}} \geq \ \frac{2 |E(G)|}{2m} \
\lambda^{^{C_{2n+1}}}_{_{2n+1}},$$ where $\lambda^{^{G}}_{_{m}}$
and $\lambda^{^{C_{2n+1}}}_{_{2n+1}}$ stand for the largest
eigenvalues of Laplacian matrices of $G$ and $C_{2n+1}$,
respectively.
\end{alphthm}

\section{Concluding Remarks}
It is instructive to add some notes on the whole setup we have
introduced so far. It is evident from our approach that any kind
of information about power thickness of a graph has important
consequences on graph homomorphism problem. There are several
questions about power thickness which remain open. In fact, we
don't know whether the power thickness is always a rational
number.

\begin{qu}\label{RATI}
Let $G$ be a non-bipartite graph and $i \geq -\chi(G)+3$ be an
integer. Is $\theta_i(G)$ a rational number? Also, for which real
number $r > 1$ there exists a graph $G$ with $\theta_i(G)=r$.
\end{qu}

Finally, we consider the following parameter as a natural
generalization of power thickness and as a measure for graph
homomorphism problem.

\begin{defin}{Let $G$ and $H$ be two graphs. Set
$$\theta_{_{H}}(G) \isdef \sup\{{2r+1 \over 2s+1}| G^{2r+1 \over 2s+1} \longrightarrow H,\ {2r+1 \over 2s+1}< og(G) \}.$$}
\end{defin}

It is easy to show that for any non-bipartite graphs $G$ and $H$,
$\theta_{_{H}}(G)$ is a real number. Also, it is obvious to see
that there is a homomorphism from $G$ to $H$ if and only if
$\theta_{_{H}}(G)\geq 1$.\\
\ \\
{\bf Acknowledgement:} The authors wish to thank M. Alishahi and
M. Iradmusa for their useful comments.

\end{document}